\documentclass[reqno]{amsart}

\usepackage{latexsym}
\usepackage{amssymb,amsthm,amsmath}

\usepackage[dvips]{graphicx}

\def\R{{\bf{R}}}

\def\N{\bf{N}}
\def\C{\bf{C}}

\def\cal{\mathcal}

\def\F{{\cal F}}

\def\pa{\partial}
\def\ph{\varphi}
\def\vp{\varphi}
\def\b{\infty}
\def\wc{\rightharpoonup}

\def\scon{\relbar\joinrel\rightarrow}

\def\be{\begin{equation}}
\def\ee{\end{equation}}
\def\ds{\displaystyle}
\def\eps{\varepsilon}
\def\SE{\cal{SE} (\R^d\times \textbf{S}^{d-1})}
\def\Sd{{\bf{S}}^{d-1}}

\newtheorem{theorem}{Theorem}[section] 

\newtheorem{proposition}[theorem]{Proposition}

\title[$H$-distributions via Sobolev spaces]
 {$H$-distributions via Sobolev spaces} 

\author{J.~Aleksi\'c, S.~Pilipovi\'c and I.~Vojnovi\'c}

\begin{document}
\maketitle

\begin{abstract}
H-distributions associated to weakly convergent sequences in Sobolev spaces are determined.
It is shown that a weakly convergent sequence $(u_n)$ in $W^{-k,p}( \R^d)$ has the property that $\theta u_n$ converges strongly in $W^{-k,p}(\R^d)$ for every $\theta\in\mathcal S(\R^d)$ if and only if all  H-distributions related to this sequence are equal to zero. 
Results are applied  on a weakly convergent sequence of solutions to a family of linear first order PDEs.
\end{abstract}


\section{Introduction}

\emph{H-measures},  or \textit{Microlocal defect measures}, of Tartar \cite{Tar} and  G\'erard \cite{Ger}
obtained  for weakly convergent sequences in $L^2(\R^d)$, and their generalization to $L^p(\R^d)$, $p\in(1,\b)$,  called H-distributions \cite{NADM}, are widely used to determine whether a weakly convergent sequence of solutions to certain classes of equations converges strongly. For example, by using H-measures the authors of \cite{AMP} obtained $L^1_{\rm loc}$-precompactness of solutions to diffusion-dispersion approximation for a scalar conservation law. In homogenization theory applications of these objects can be found e.g. in \cite{Ant3} and \cite{Mie}. In \cite{Pan}, H-measures are applied to family of entropy solutions of a first order quasilinear equation and in  \cite{Sazenkov} to ultraparabolic equation. The list of applications of these objects is far from being complete.

Our aim in this paper is to extend the concept of H-distributions  to the Sobolev spaces.  
For the purposes of this paper, we introduce in  Subsection \ref{ds} new tensor product - spaces of test functions and distributions. For the reader's convenience, we give
full description of such spaces in the Appendix (Propositions \ref{pr1} and \ref{pr2}).

In order to use the duality $W^{-k,p}$-$W^{k,q}$, $q=\frac p{p-1}$, $k\in{\N}_0$, we prove the existence result for H-distributions associated to a weakly convergent sequence in $L^p(\R^d)$;
in  Theorem \ref{prva} we extend the result of  \cite[Theorem 2.1]{NADM} since we did not use the localization coming from the compactly supported test functions.
H-distributions of Theorem \ref{prva} are defined on  the space of rapidly decreasing functions. This leads to  the improvements of results of \cite{NADM}  in the case of $L^p-$spaces.  In 
Theorem \ref{tSP} we prove the existence of   H-distributions for weakly convergent sequences in Sobolev spaces.
Our main theorem, Theorem \ref{th7}, shows that
if for a given weakly convergent sequence $u_n\wc 0$ in $W^{-k,p}(\R^d)$ and every weakly convergent sequence $v_n\wc 0$ in $W^{k,q}(\R^d)$  the corresponding H-distributions are  equal to zero, then for every $ \ph\in{\cal S}(\R^d)$, $(\vp u_n)$ converges strongly to zero in $W^{-k,p}(\R^d)$.
Clearly, the converse assertion also holds. As an application, we analyze in Theorem \ref{th8}  a weakly convergent sequence $(u_n)$ 
of solutions to $\sum_{i=1}^d \pa_i  \left( A_i(x)u_n \right)=f_n$  in $W^{-k,p}(\R^d), d>\frac{p}{p-1}$,
and show that the supports of the  corresponding H-distributions are concentrated on the characteristic set 
$\{(x,\xi): \sum_{i=1}^{d}A_i(x)\xi_i=0\},$ under the new condition that for every $\vp\in\cal S(\R^d)$, $(\vp f_n)$ strongly converges to zero in 
$W^{-k-1,p}(\R^d)$. Moreover, if all H-distributions assigned to this equation are equal to zero, then $(\vp u_n)$ converges strongly to zero in
$W^{-k,p}(\R^d)$. 
The corresponding results for $L^2(\R^d)$ and $L^q(\R^d)$ are obtained in \cite{Tar} and  \cite{NADM}, respectively. Even in mentioned  cases our results  with $k=0$ extend previous  results since the non-locality is the essential part of our approach. Moreover, 
the results from the recent contributions in which the H-distributions were used cf. \cite{Lazar,Misur,Rindler} can be extended to a more general situations (in the Sobolev spaces with negative coefficients) by using results from this paper.

\section{Basic definitions and assertions}

\subsection{Some spaces of distributions }\label{ds}
We refer to \cite{Adams} for the Sobolev spaces $W^{k,q}(\R^d)$. 
If $k>\frac dq$,  then $W^{k,q}(\R^d)\subset C_0(\R^d)$, where $ C_0(\R^d)$ is the space of continuous functions vanishing at infinity. The dual $\left(W^{k,q}(\R^d)\right)' = :W^{-k,p}(\R^d)$ is isometrically isomorphic to the Banach space consisting of distributions $u\in {\cal S}'(\R^d) $ of the form $u=\ds\sum_{|\alpha|\leq k}\pa^\alpha u_\alpha$, where  all $u_\alpha \in L^p(\R^d)$, normed by $\|u\|:=\inf\Big\{ \Big( \ds\sum_{|\alpha|\leq k} \|u_\alpha\|^p_p \Big)^{1/p}\,:\, u=\ds\sum_{|\alpha|\leq k}\pa^\alpha u_\alpha \Big\},$ cf. \cite[Theorem 3.10, p. 50]{Adams}.

In order to give clear explanations concerning a new space, which will be denoted by $\cal{SE} (\R^d\times \textbf{S}^{d-1})$, and its dual $\cal{SE}' (\R^d\times \textbf{S}^{d-1})$, we will use some classical results, \cite{RS} and \cite{Aubin},  of $L^2$ and Sobolev theory  
for the 
unit sphere $\textbf{S}^{d-1}$ as well as of \cite{Atkinson} for some results for $C^k$ and $L^2$ functions on $\Sd$. Concerning Sobolev spaces and distributions on a manifold, we refer to \cite{Su} and for tensor product of test spaces, to \cite{TrevesTVS}.

We define the space of smooth functions  $\SE$ by the sequence of norms 
\begin{equation}\label{d1}
p^\infty_{\R^d\times\Sd, k}(\theta ) = \sup_{(x,\xi)\in\R^d\times\Sd, |\alpha+\beta |\leq k} \langle x\rangle^k |(\Delta^\star_\xi)^\alpha \pa_x^\beta \, \theta(x,\xi)|,
\end{equation}
where $\langle x\rangle^k=(1+|x|^2)^{k/2}$ and $\Delta^\star$ is the Laplace-Beltrami operator. 
The space $\SE$ is a Fr\' echet space and can be identified with the completion of tensor product $\mathcal S({\bf R}^{d})  \hat{\otimes}   \mathcal E({\bf S}^{d-1}) $, as was shown in 
Proposition \ref{pr2} in the Appendix. Complete description of this space can be found in the  Appendix.

\subsection{H-distributions on $L^p$ spaces}

A bounded function $\psi$, on $\R^d$,  is called  $L^p$-Fourier multiplier   if $ f\mapsto {\cal A}_\psi(f):=(\psi\hat{f})^{\check{}}$ is a bounded mapping from $\cal S(\R^d)$ to $L^p(\R^d)$ and can be continuously extended to a mapping from $L^p(\R^d)$ to $L^p(\R^d)$.
Here $\ds\hat f(\xi) = {\cal F}[f](\xi) =\int_{\R^d} e^{-2\pi ix\cdot\xi} f(x)\,dx $ denotes the Fourier transform on $\R^d$, while  
$\ds\check g(x) = {\cal F}^{-1}[g](x) =\int_{\R^d} e^{2\pi ix\cdot\xi} g(\xi)\,d\xi $ 
denotes the inverse Fourier transform.
The space of $L^p$-Fourier multipliers, denoted by ${\cal M}_p(\R^d)$, $1< p< \b$  (cf. \cite{Gra}), is  supplied by the norm $\|\psi\|_{\cal M_p}:=\|{\cal A}_\psi\|_{L^p\to L^p},$ where $\|\cdot\|_{L^p\to L^p}$ is the standard operator norm.

If $\psi\in C^\kappa(\R^d \backslash \{0\})$, $\kappa =  [\frac d2]+1$,  is homogeneous of zero degree (i.e. $\psi(\lambda\xi)=\psi(\xi)$, $\lambda>0$), then $\psi\in L^\infty(\R^d)$ and 
\be\label{Mih}
|\pa_\xi^\alpha\psi(\xi)| \leq A |\xi|^{-|\alpha|}, \quad \xi\in \R^d \backslash \{0\},
\ee
for every $|\alpha|\leq \kappa$ (with  $A= \ds \max_{|\beta|\leq \kappa} \sup_{\xi\neq 0} |\xi|^\alpha |\pa^\beta \psi|$, cf. \cite[p. 120]{AH}).  Thus $\psi$ fulfills  conditions from the Mihlin theorem (cf. \cite{Gra}):
\emph{Let $\psi$ be a complex-valued bounded  function on $\R^d\backslash \{0\}$ that satisfies  (\ref{Mih})
     for all multi-indices $|\alpha| \leq [\frac d2]+1$.    
     Then $\psi\in {\cal M}_p(\R^d)$  for any  $1<p<\b$ and
\begin{equation}\label{Mih2}
\|\psi\|_{{\cal M}_p}\leq C_d\max \left\{p, \frac1{p-1} \right\} (A+\|\psi\|_\b).
\end{equation}}
Moreover, if $\psi\in C^\kappa ({{\bf{S}}}^{d-1})$, 
then constant $A$ in (\ref{Mih2}) can be replaced by 
$\|\psi\|_{C^\kappa({\bf{S}}^{d-1})}$. 

Fourier multiplier operators ${\cal A}_\psi$  with symbol $\psi\in C^{\kappa}({\bf{S}}^{d-1})$ can be defined on $W^{-k,p}(\R^d)$, via duality
$$
\,_{ W^{-k,p}}\langle {\cal A}_\psi u, v \rangle_{ W^{k,q}} := \,_{ W^{-k,p}}\langle u, {\cal A}_{\bar\psi} v \rangle_{ W^{k,q}}.
$$
Since  $\pa^\alpha {\cal A}_{\bar\psi} v= {\cal A}_{\bar\psi}(\pa^\alpha v)$, we know that ${\cal A}_{\bar\psi} v \in  W^{k,q}(\R^d)$.
If 
$u\in W^{-k,p}(\R^d)$ is 
of the form $u=\ds\sum_{|\alpha|\leq k}\pa^\alpha u_\alpha$, then for all $v\in W^{k,q}(\R^d)$,
\begin{eqnarray}
\nonumber\ds
&\ds\,_{ W^{-k,p}}\langle {\cal A}_\psi u, v \rangle_{ W^{k,q}} = \sum_{|\alpha|\leq k} \,_{ W^{-k,p}}\langle \pa^\alpha u_\alpha , {\cal A}_{\overline\psi} v \rangle_{ W^{k,q}} =\\
& \ds =\nonumber
\sum_{|\alpha|\leq k} (-1)^{|\alpha|}  \,_{L^p}\langle u_\alpha,  {\cal A}_{\overline\psi} (\pa^{\alpha} v) \rangle_{L^q} =
\sum_{|\alpha|\leq k} (-1)^{|\alpha|}  \,_{L^p}\langle {\cal A}_\psi (u_\alpha), \pa^\alpha v \rangle_{L^q}.
\end{eqnarray}
One can see that every $L^p$-multiplier operator ${\cal A}_\psi$ with symbol  $\psi\in {\cal M}_p(\R^d)$ is a bounded operator from $W^{-k,p}(\R^d)$ to  $ W^{-k,p}(\R^d)$.

In order to prove the existence of an H-distributions of Theorem \ref{tSP} given below, we need   Tartar's First commutation 
lemma \cite{Tar} and the modification of this lemma given in \cite{NADM} .

\textit{\cite{Tar}: 
Let $\psi\in C({\bf{S}}^{d-1})$ and $b\in C_0(\R^d)$ define the Fourier multiplier operator ${\cal A}_\psi$ and the operator of multiplication $B$, acting  on $u\in L^2(\R^d)$, as follows: 
$\F({\cal A}_\psi u)(\xi)=\psi\Bigl(\frac{\xi}{|\xi|}\Bigr)\F(u)(\xi)$, $\xi\in\R^d\backslash\{0\}$, and $Bu(x)=b(x)u(x)$, $x\in\R^d$.
Then the operators ${\cal A}_\psi$ and $B$ are bounded on $L^2(\R^d)$, and  their commutator $C:={\cal A}_\psi B-B{\cal A}_\psi$ is a compact operator from $L^2$ into itself.}

 \textit{Moreover, \cite{NADM}:
  If a sequence $(v_n)$ is bounded in both $L^2(\R^d)$ and
$L^r(\R^d)$, for some $r\in(2,\b]$ and $v_n\rightharpoonup 0$ in the sense
of distributions, then the sequence $(Cv_n)$ strongly converges to
zero in $L^q(\R^d)$, for any $q\in [2,r]\backslash\{\b\}$.}

\begin{theorem}\label{prva}
If $u_n \wc 0$ in $L^{p}({\R}^{d})$, and $v_n \wc  0$ in $L^q(\R^{d})$, then there exist subsequences $(u_{n'})$, $(v_{n'})$ and a  distribution $\mu(x,\xi)\in {\cal{SE}}'(\R^{d} \times {\bf{S}}^{d-1})$ of order not more than $\kappa=[d/2]+1$ in $\xi$, such that for every $\varphi_1, \varphi_2\in {\cal S}(\R^d)$ and $\psi\in C^{\kappa}({\bf{S}}^{d-1})$,
\begin{equation}
\begin{array}{cc}\label{Te_AM}
& \displaystyle \lim\limits_{n'\to \infty}\!\int_{\R^{d}}\!{{\cal A}_{\psi}(\varphi_1 u_{n'})(x)}\overline{(\varphi_2 v_{n'})(x)}dx\!
 =\!\lim\limits_{n'\to \infty}\!\int_{\R^{d}}\!(\varphi_1 u_{n'})(x)\overline{{\cal A}_{\overline{\psi}}(\varphi_2 v_{n'})(x)}dx\\
& =: \langle\mu ,\varphi_1\overline\varphi_2\psi \rangle,
\end{array}
\end{equation}
where ${\cal A}_{\psi}:L^p(\R^d)\to L^p(\R^d)$ is a Fourier multiplier operator with the symbol $\psi\in C^\kappa({\bf{S}}^{d-1})$.
\end{theorem}

By the order of $\mu\in\SE$ we mean that for any $\varphi\in {\cal S} (\R^d)$,
$\langle \mu(x,\xi), \varphi(x) \psi(\xi)\rangle$ can be extended on 
$C^\kappa(\textbf{S}^{d-1})$
(see (\ref{zvez}) and (\ref{zvezzvez}) below).

\begin{proof}
First, notice that the Fourier multiplier operator ${\cal A}_\psi$ with $\psi\in C^\kappa({\bf{S}}^{d-1})$
 is well defined on both 
 $\varphi_1 u_{n}\in L^p(\R^d)$ and 
 $\varphi_2 v_{n}\in L^q(\R^d)$,
 and that the adjoint operator of ${\cal A}_{\psi}$ is ${\cal A}_{\overline\psi}$. Thus, the first equality in (\ref{Te_AM}) holds.

Let $1<p\leq 2$. Consider a sequence of sesquilinear (linear in $\psi\in C^{\kappa}(\mathbf{S}^{d-1})$ and  anti-linear in $\vp\in\cal S(\R^d)$) functionals 
\be\label{pre12}
\mu_n (\vp,\psi)= \int_{\R^d} u_n \overline{{\cal A}_{\overline{\psi}}(\vp v_n)}dx.
\ee
By the continuity of $\cal A_\psi$ and the boundedness of $(u_n)$ and $(v_n)$ in $L^p(\R^d)$ and $L^q(\R^d)$, it follows that
 there exists $c>0$,  such that for every $n\in{\N}$,
\be\label{00}
|\mu_n (\vp,\psi)|\leq \|u_n\|_{L^p} \|{\cal A}_{\overline{\psi}}(\vp v_n)\|_{L^q}\leq c\,\|\psi\|_{C^\kappa({\bf{S}}^{d-1})} \|\vp\|_{L^\b}.
\ee
Fix $\vp\in{\cal S}(\R^{d})$ and denote by $(B_n \vp)$  the sequence of functions defined on $C^\kappa({\bf{S}}^{d-1})$  by
\be\label{broj}
\langle B_n \vp,\cdot\rangle= \mu_n (\vp, \cdot ).
\ee
For every $n\in {\N},$  the linearity of $B_n \vp$ is clear and the continuity follows from (\ref{00}):
\be\label{ok}
|\langle B_n \vp,\psi\rangle | \leq c_\vp\,\|\psi\|_{C^\kappa({\bf{S}}^{d-1})},  \mbox{ where } c_\vp=c||\vp||_{L^\infty}.
\ee
If we  fix $\psi \in C^\kappa({\bf{S}}^{d-1})$, then  (\ref{pre12}) implies that  the mapping 
$ {\cal S} (\R^d) \to \C$, $\vp\mapsto\langle B_n\vp,\psi\rangle$ is anti-linear and, again by (\ref{00}), continuous. 

We continue with fixed $\vp$ and apply the Sequential Banach Alaoglu theorem to obtain weakly star convergent subsequence
$(B_k \vp)$ in $ (C^\kappa({\bf{S}}^{d-1}))'$. We denote the weak star limit of $B_k \vp$ by $B\vp$, i.e. for every $\psi \in C^\kappa({\bf{S}}^{d-1})$,
$$
\langle B\vp, \psi\rangle = \lim_{k\to\b} \langle B_k \vp, \psi \rangle .
$$
We are going to show that $B$ can be defined on the whole $\cal S(\R^d)$, so that 
$\cal S(\R^d)\ni\vp\mapsto B\vp \in (C^\kappa ({\bf{S}}^{d-1}))'$ is linear and continuous.

By the diagonalization argument, we define $B$  
on a countable dense set $M=\{ \vp_m |\, m\in {\N} \}\subset \cal S(\R^d)$.
For that purpose extract a subsequence $(B_{1,k})_{k} \subset (B_{n})_{n}$ such that
 $(B_{1,k} \vp_1)$ is weakly star convergent in $(C^\kappa({\bf{S}}^{d-1}))'$ and denote the limit as $B\vp_1$. Then extract a subsequence $(B_{2,k})_{k} \subset (B_{1,k})_{k} $ such that $(B_{2,k} \vp_2)$ is weakly star convergent in $(C^\kappa({\bf{S}}^{d-1}))'$ and denote the limit as
 $B\vp_2$. Notice  also that
  $B_{2,k} \vp_1$ converges weakly star to $B\vp_1.$ 
  Repeating this procedure (extracting subsequences for all $\vp_m\in M$), we obtain diagonal (sub)sequence $B_{k,k}\in\cal L
\left(
\cal S(\R^d), 
\left(C^\kappa
({\bf{S}}^{d-1})\right)'
\right)$,  such that for all $\vp_m\in M$
$$
\langle B\vp_m, \psi\rangle = \lim_{k\to\b} \langle B_{k,k} \vp_m, \psi \rangle, \quad \psi \in C^\kappa({\bf{S}}^{d-1}).
$$
Denote $B_{k,k} =:b_{k}$ 
and 
fix $\psi\in C^\kappa ({\bf{S}}^{d-1})$. By (\ref{broj}),   $\vp\mapsto \langle b_k\vp,\psi\rangle,$ is a pointwise bounded sequence in $\cal S'(\R^d)$ which converges
on a dense set $M\subset\cal S(\R^d)$. By the Banach-Steinhaus theorem, see e.g. \cite[p. 169]{Kothe}, $\langle b_k (\cdot),\psi\rangle$ converges to $\langle B(\cdot),\psi\rangle$ on $\cal S(\R^d)$.
In this way we show that for every $\vp\in\cal S(\R^d)$ and every  $\psi\in C^\kappa({\bf{S}}^{d-1})$
$$
\lim_{k\rightarrow \infty}\langle b_k \vp, \psi\rangle =\langle B\vp, \psi\rangle .
$$
Moreover, by (\ref{broj}),
\begin{equation}\label{zvez}
|\langle B\vp,\psi\rangle|\leq c||\vp||_{L^\infty}||\psi||_{C^\kappa({\bf{S}}^{d-1})}.
\end{equation}
By  \cite[Part III, Chap. 50, Proposition 50.7, p. 524]{TrevesTVS} (it is a version of the Schwartz kernel theorem)  we have that there exists 
$\mu\in \cal{SE}'(\R^d\times{\bf{S}}^{d-1})$ defined as

\begin{equation}\label{zvezzvez}
\langle \mu(x,\xi), \vp(x)\psi(\xi)\rangle=\lim_{k\to\b} \langle b_k\vp,\psi\rangle = \lim_{k\to\b} \int u_{k} \overline{{\cal A}_{\overline{\psi}}(\vp v_{k})}dx, 
\end{equation}
for all $\vp\in{\cal S}(\R^{d}),\;\psi \in C^\kappa({\bf{S}}^{d-1}),$ where 
$(u_k)$ is a subsequence of $(u_n)$ and $(v_k)$ is a subsequence of $(v_n)$   corresponding to $(b_k)$. Now, we will use the factorization property of $\cal S(\R^d)$, \cite{SP}: Every $\vp\in \cal S(\R^d)$ can be written as $\vp=\overline\vp_1 \vp_2$, for some $\vp_1,\vp_2\in \cal S(\R^d)$. Then
$$
\langle \mu, \vp\psi\rangle = \lim_{k\to\b} \int u_{k} \overline{{\cal A}_{\overline{\psi}}(\overline{\vp}_1 \vp_2 v_{k})}dx. 
$$
Since $\| \vp_2v_k\|_{L^2}\leq \| v_k\|_{L^q} \|\vp_2\|_{L^{\frac{2q}{q-2}}}$, we can apply the commutation lemma  to $\vp_2v_k\in L^2\cap L^q$ and $\overline\vp_1\in {\cal S}(\R^d) \subset C_0(\R^d)$
to obtain that  for every $\vp_1,\vp_2\in \cal S(\R^d)$ and $\psi\in C^\kappa({\bf{S}}^{d-1})$,
\begin{equation}\label{n1}
\langle \mu, \overline{\vp}_1\vp_2\psi\rangle = \lim_{k\to\b} \int_{\R^d} \vp_1 u_{k}  \overline{ {\cal A}_{\overline{\psi}}(\vp_2 v_{k})} dx .
\end{equation}
This completes the proof of Theorem \ref{Te_AM} for $1<p\leq 2$.

In the case when $p> 2$, we define 
$$
\mu_{n}(\vp,\psi):=
 \int_{\R^{d}}   {\cal A}_{\psi} (\vp u_{n})  \overline{ v}_{n}  dx.
$$
Then, in the same way as above, but now with the change of the roles of $(u_n)$ and $(v_n)$, we use factorization $\vp=\vp_1\overline{\vp}_2$,  then the commutation lemma on $\vp_1 u_{n}\in L^2(\R^d)$  
and apply the preceding proof.
\end{proof}

\textbf{Remark:}

 The formulation  of the previous theorem can be slightly changed in the case when $p\in(1,2)$. Then, instead of $v_n\wc 0$ in $L^q$ we can assume that  $v_n \wc 0$ in $L^r$ for some $r\geq q$ and obtain the same result as in Theorem \ref{prva}. In that case for every $\ph\in {\cal S}(\R^d)$, $\ph v_n \in L^q\cap L^2$ and the same proof can be applied.  
The same idea but with compactly supported $\ph$  was used in \cite{NADM}.

\section{H-distribution and Sobolev spaces} 

The next theorem determines H-distributions associated to sequences in  Sobolev space.

\begin{theorem}\label{tSP}
If a sequence $u_n  \wc  0$ weakly  in $ W^{-k,p}(\R^d)$ and $v_n \wc 0$ weakly in $W^{k,q}(\R^d)$,  then there exist subsequences $(u_{n'}), (v_{n'})$ and a distribution $\mu\in \cal{SE}'(\R^d\times{\bf{S}}^{d-1})$ such that for every $\ph_1,\ph_2 \in {\cal S} (\R^d)$
and every $\psi\in C^{\kappa}({\bf{S}}^{d-1})$,
\begin{equation}\label{3}
\lim_{n'\to\b} \langle {\cal A}_\psi (\ph_1 u_{n'})\, , \, {\ph_2 v_{n'}} \rangle  =\lim_{n'\to\b} \langle \ph_1 u_{n'}\, , \, {{\cal A}_{\overline{\psi}} (\ph_2 v_{n'})} \rangle =\langle \mu,\ph_1\bar\ph_2 \psi\rangle.
\end{equation}
\end{theorem}

\begin{proof}
Since $u_n \wc 0$ in $W^{-k,p}(\R^d)$, there exist a subsequence $u_{n'} \wc 0$ such that  $u_{n'}=\ds\sum_{|\alpha| \leq k}   \pa^\alpha g_{\alpha, n'},$ 
where for every $|\alpha|\leq k$,  $(g_{\alpha, n'})$ is a sequence of $L^p$-functions such that  
$g_{\alpha,n'} \wc 0$ in $L^p(\R^d)$. 
Indeed, since a weakly convergent sequence forms a bounded set in $W^{-k,p}(\R^d),$ using the same proof of the representation theorem for elements of $W^{-k,p}(\R^d),$ one can
 obtain the existence of bounded sets $\{F_{\alpha,n}, n\in{\N}\}$, $|\alpha|\leq k$, such that $u_{n}=\ds\sum_{|\alpha| \leq k}   \pa^\alpha F_{\alpha, n}.$ 
Now, since $\{F_{\alpha,n}, n\in{\N}\}$ are bounded in $L^p(\R^d)$, these sets are weakly precompact and every $\{F_{\alpha,n}, n\in{\N}\}$ has a weakly convergent subsequence. By the diagonalization method one can find a subsequence such that
   $F_{\alpha,n'}\wc f_\alpha\in L^p(\R^d)$, $n'\to \b$, $|\alpha|\leq k$, in $L^p(\R^d)$. 
Since 
$\ds\sum_{|\alpha| \leq k}   \pa^\alpha F_{\alpha, n'} \wc 0$, it follows that   $\ds\sum_{|\alpha|\leq k} \pa^\alpha f_\alpha=0$. Thus we obtain required subsequence $u_{n'}= \ds\sum_{|\alpha|\leq k} \pa^\alpha (F_{\alpha,n'}-f_\alpha)$. 
In the sequel we will not relabel subsequences, so we will use $u_n$ instead of $u_{n'}$.

Since  
$$
\pa^\alpha_x \Big[{\cal A}_\psi(u)\Big] = {\cal A}_{\psi_\alpha} (u) = {\cal A}_{\psi} (\pa^\alpha u), \mbox{  for } {\psi_\alpha}(\xi)= (2\pi i)^{|\alpha|} \xi^\alpha \psi(\xi),
$$
we have that
$$
{\cal A}_\psi \left( \ph_1 \, \pa^\alpha F_{\alpha, n} \right)
= (-1)^{|\alpha|}\sum_{0\leq\beta\leq\alpha} (-1)^{|\beta|} \left(\begin{array}{c} \alpha\\ \beta
\end{array}
\right) 
 \, \pa^\beta \! \left[ {\cal A}_\psi \left(\, F_{\alpha, n}\,  \pa^{\alpha-\beta}\ph_1 \right) \right],
$$ 
and so
\begin{equation}
\begin{array}{cc}
& \left\langle {\cal A}_\psi (\ph_1 u_{n})\, , \, {\ph_2} {v_{n}} \right\rangle = \sum_{|\alpha|\leq k}
           (-1)^{|\alpha|} \sum_{0\leq\beta\leq\alpha} 
           \left(\begin{array}{c} \alpha \\ \beta \end{array} \right) 
           \left\langle {\cal A}_\psi \left( F_{\alpha, n}\, \pa^{\alpha-\beta} \ph_1  \right)\, , \,
             \pa^{\beta}\![{\ph_2} {v_{n}}] \right\rangle \\
           & \ds \qquad = \sum_{|\alpha|\leq k} (-1)^{|\alpha|} \sum_{0\leq\beta\leq\alpha} \left(\begin{array}{c} \alpha \\ \beta \end{array} \right)  \sum_{0\leq \gamma \leq \beta}  \left(\begin{array}{c} \beta \\ \gamma \end{array} \right) \left\langle {\cal A}_\psi
             \left( F_{\alpha, n} \, \pa^{\alpha-\beta} \ph_1 \right)\, , \,  \pa^{\beta - \gamma} \ph_2 \,
            \pa^\gamma  v_{n}  \right\rangle.
\end{array}
\end{equation}

For the moment, we fix $\alpha$ and  apply  Theorem \ref{prva} to $F_{\alpha,n} \wc 0$ in $L^p(\R^d)$ and $v_{n}  \wc   0$  in ${L^q}(\R^d)$, thus obtaining subsequences $(F_{\alpha, n_0})_{n_0}$, $(v_{\alpha, n_0})_{n_0}$ and an H-distribution
$\mu_{\alpha, 0}\in {\cal{SE}}'(\R^d \times {\bf{S}}^{d-1}),$ such that 
$$
\left\langle \, \mu_{\alpha, 0}\,,\, \varphi_1 \overline{\varphi}_2 \psi \right\rangle := \lim_{n_0\to\b} \left\langle {\cal A}_\psi \left( \varphi_1 F_{\alpha, n_0} \right), \,{ \varphi_2} \, {v_{\alpha, n_0}} \right\rangle.
 $$
Then, applying Theorem \ref{prva}  to $F_{\alpha, n_0} \wc 0$ in $L^p(\R^d)$,  and $\pa^{(1,0,...,0)} v_{\alpha, n_0}  \wc  0$ 
 in ${L^q}(\R^d)$,
we obtain subsequences $(F_{\alpha, n_{(1,0,...,0)}})_{n_{(1,0,...,0)}}$, $(v_{\alpha, n_{(1,0,...,0)} })_{n_{(1,0,...,0)}}$ and an H-distribution
$\mu_{\alpha, (1,0,...,0)}$ $\in {\cal{SE}}'(\R^d \times {\bf{S}}^{d-1}).$  Thus, we obtain finitely many H-distributions  $\mu_{\alpha,\gamma}$, $0\leq \gamma \leq \alpha$,
such that
$$
\left\langle \, \mu_{\alpha,\gamma}\,,\, \varphi_1 \overline{\varphi}_2 \psi \right\rangle := \lim_{n_\gamma\to\b} \left\langle {\cal A}_\psi \left( \varphi_1 F_{\alpha, n_\gamma} \right) , \, { \varphi_2} \, { \pa^\gamma v_{\alpha, n_\gamma}} \right\rangle.
$$
The last one $\mu_{\alpha, \alpha}$ is obtained together with subsequences $(F_{\alpha, n_\alpha})_{n_\alpha}$, $(v_{\alpha, n_\alpha})_{n_\alpha}$ which we are going to use to  define H-distribution $\mu^\alpha$ in the following way: For  $\varphi_1, \varphi_2\in {\cal S}(\R^d),$ $\psi \in C^\kappa({\bf{S}}^{d-1}),$ 
$$
             \left\langle \, \mu^\alpha \,,\, \ph_1 \overline{\ph}_2 \psi \right\rangle :=  (-1)^{|\alpha|} \sum_{0\leq\beta\leq\alpha} \left(\begin{array}{c} \alpha \\ \beta \end{array} \right)  \sum_{0\leq \gamma \leq \beta}  \left(\begin{array}{c} \beta \\ \gamma \end{array} \right)  \left\langle \mu_{\alpha, \gamma} \, , \,  \pa^{\alpha-\beta}\ph_1\,  \pa^{\beta-\gamma}\bar{\ph}_2 \, \psi
                 \right\rangle.
$$
The sum on the right hand side is finite and all H-distributions $\mu_{\alpha,\gamma}$ can be defined via $(F_{\alpha, n_\alpha})_{n_\alpha}$ which is subsequence of $ F_{\alpha, n_\gamma}$ and $(v_{\alpha, n_\alpha})_{n_\alpha}$ which is subsequence of $(v_{\alpha, n_\gamma})_{n_\gamma}$,  so the H-distribution $\mu^\alpha$ is well-defined.

Let us emphasize that we have obtained $\mu^\alpha$ for a fixed $\alpha$. Now if we take first $\alpha=0$ with previous procedure we can obtain H-distribution $\mu^0$ defined via $(F_{0,n_0})_{n_0}$ and $(v_{n_0})_{n_0}$. Then, starting with $(F_{e_1,n_0})_{n_0}$ and $(v_{e_1,n_0})_{n_0}$ we obtain (by the same procedure)  H-distribution $\mu^{e_1}$ defined via $(F_{e_1,n_{e_1}})_{n_{e_1}}$ and $(v_{e_1, n_{e_1}})_{n_{e_1}}$. Here $e_1=(1,0,...,0)$. Then we proceed  with $e_2=(0,1,0,...,0)$ to obtain H-distribution $\mu^{e_2}$ and so on with all $|\alpha|\leq k$. 

At the end we obtain H-distribution $\mu$ defined by
$$
             \left\langle \, \mu  \,,\, \ph_1 \overline{\ph}_2 \psi \right\rangle :=  \sum_{|\alpha|\leq k} (-1)^{|\alpha|} \sum_{0\leq\beta\leq\alpha} \left(\begin{array}{c} \alpha \\ \beta \end{array} \right)  \sum_{0\leq \gamma \leq \beta}  \left(\begin{array}{c} \beta \\ \gamma \end{array} \right)  \left\langle \mu_{\alpha, \gamma} \, , \,  \pa^{\alpha-\beta}\ph_1\,  \pa^{\beta-\gamma}\bar{\ph}_2 \, \psi
                 \right\rangle.
$$
and subsequences $\left( \ds\sum_{|\alpha|\leq k} \pa^\alpha F_{\alpha, n_{(0,...,0,k)}}\right)_{n_{(0,...,0,k)}}$ and $\left( \ds v_{n'}\equiv v_{(0,...,0,k), n_{(0,...,0,k)}}\right)_{n_{(0,...,0,k)}}$.
\end{proof}

Distribution $\mu$ obtained in Theorem \ref{tSP}  is called  {\em   $H$-distribution} corresponding to the (sub)sequence $(u_{n'},v_{n'})$. 

$\,$

Assume that the distributions $\mu$ determined by Theorem \ref{tSP} are equal to zero. Then the local strong convergence in $W^{-k,p}(\R^d)$ easily follows. We will prove a more delicate assertion in the next theorem.

\begin{theorem}\label{th7}
 Let 
 $u_n \wc 0$ in $W^{-k,p}(\R^d)$. If for every sequence $v_n \wc 0$ in $W^{k,q}(\R^d)$ 
 the corresponding H-distribution is zero, then for every $\theta\in {\cal S} (\R^d)$, $\theta u_n \to 0$ strongly in $W^{-k,p}(\R^d)$, $n\to\b$.
\end{theorem}

\begin{proof}
For the strong convergence we need to prove that for every $\theta\in {\cal S} (\R^d)$
\[
{\sup}\{\langle \theta u_n , \phi \rangle \,:\,\phi\in B \} {\scon} 0, \; n\to \b, 
\mbox{ for every bounded } B\subseteq W^{k,q}(\R^d).
\]
If it would  not be  true, then there would exist $\theta \in {\cal S} (\R^d)$,  a bounded set $B_0$  in $W^{k,q}(\R^d)$, an $\eps_0>0$ and a subsequence $(\theta u_k)\subset (\theta u_n)$, such that 
$$ 
{\sup} \{ \left | \langle \theta u_k , \phi \rangle \right| \, : \, \phi\in B_0 \}  \geq \eps_0, \mbox{ for every } k\in {\N}.
$$
Choose $\phi_k\in B_0$ such that  $\ds | \langle \theta u_k , \phi_k \rangle| > {{\eps_0} / 2}$. 
Since $\phi_k\in B_0$ and $B_0$ is bounded in $W^{k,q}(\R^d)$,  $(\phi_k)$ is weakly precompact in $W^{k,q}(\R^d)$, i.e. up to a subsequence, $\phi_k \wc \phi_0$ 
in $W^{k,q}(\R^d)$. 
Moreover, since $\phi_0$ is fixed, $\ds\langle u_k , \phi_0 \rangle \to 0$ 
and 
\be \label{kontr2}
        | \langle \theta u_k , \phi_k -\phi_0 \rangle | > {\frac{\eps_0}4}, \quad k>k_0.
\ee
Applying Theorem \ref{tSP} on $u_k \wc  0$ and $\phi_k - \phi_0 \wc 0$, 
we obtain that for every $\varphi_1, \varphi_2 \in {\cal S}(\R^d)$
\begin{equation}\label{kont1}
\lim_{k\to\b}    \;_{W^{-k,p}}\!\langle  {\cal A}_{\psi}(\varphi_1  u_{k}),  {\ph_2} {(\phi_k-\phi_0)} \rangle_{W^{k,q}}=0.
\end{equation}
With $\psi\equiv 1$ on ${\bf{S}}^{d-1}$, 
(\ref{kont1})  implies 
\[
\lim_{k\to\b}  \langle \ph_1 u_{k} \, , \, { \ph_2 (\phi_k-\phi_0)} \rangle = 0.
\]
Again, we use the factorization property of $\cal{S} (\R^{d})$. So if $\theta \in \cal{S}(\R^d)$,
then 
 $\theta = \phi_1\bar{\phi_2}$, for some $\phi_1, \phi_2 \in {\cal S}(\R^d)$, and  we have that  $\ds \lim_{k\to\b}  \langle \phi_1 u_{k} \, , \, { \phi_2 (\phi_k-\phi_0)} \rangle = 0$, i.e. $\ds \lim_{k\to\b}  \langle \theta u_{k} \, , \, {  (\phi_k-\phi_0)} \rangle = 0$. This contradicts (\ref{kontr2}) and completes the proof.  
 \end{proof}

\subsection{Localization property}

Recall
\cite[p. 117]{Ste}, the  Riesz potential of order $s$, $Re(s)>0$ is the operator $I_s=(-\Delta)^{-\frac s2}$., see also \cite{Gra}.
Consideration of the Fourier transform and convolution theorem reveals that $I_\alpha$, for
$0<\alpha<d$, is a Fourier multiplier, i.e.
$
{\mathcal F}[I_\alpha[ f ]](\xi):=(2\pi
|\xi|)^{-|\alpha|}{\mathcal F}[f](\xi) .
$
We will use the potential $I_1$ with the following properties:
\begin{equation}\label{poten}
 \|I_1(f)\|_{L^{\frac{qd}{d-q}}} \leq C \|f\|_{L^q}, \mbox{ for } f\in L^q(\R^d),\quad 1<q<d;
 \end{equation}
 \begin{equation}\label{2pot}
 \pa_j I_1(f) = -R_j (f), \quad f\in L^q(\R^d), \mbox{ where } R_j:={\cal A}_{\xi_j / \imath |\xi|}.
\end{equation}
Moreover,
$R_j:L^q\rightarrow L^q$ is continuous.

Consider now a sequence $u_n \wc 0$ in $W^{-k,p}(\R^d)$ satisfying the following sequence of equations: 
\be\label{jed}
          \sum_{i=1}^d \pa_i \left( A_i(x)u_n(x)\right) =f_n(x),
\ee
where $A_i\in \cal S (\R^d)$  
and $f_n$ is a sequence of temperate distributions such that
\be\label{jedjed}
       \vp f_n \to 0  \mbox{ in }  W^{-k-1,p}(\R^d),  \mbox { for every } \vp\in {\cal S}(\R^d).
\ee

\begin{theorem}\label{th8}
       Let  $1<q<d$.
If $u_n \wc 0$ in  $ W^{-k,p} (\R^d)$ satisfies (\ref{jed}), (\ref{jedjed}), then for any sequence $v_n \wc 0$ in $ W^{k,q}(\R^d)$ the corresponding H-distribution $\mu$ satisfies
\be\label{kjed}
          \sum_{j=1}^d A_j(x) \xi_j \mu(x,\xi) = 0 \quad \mbox{in } {\cal{SE}}'(\R^d\times {\bf{S}}^{d-1}).
\ee
Moreover, if (\ref{kjed}) implies $\mu(x,\xi)=0$, we have the strong convergence $\theta u_n  \scon 0$, in $W^{-k,p}(\R^d)$,  for every $\theta \in {\cal S}(\R^d)$.
\end{theorem}

\begin{proof} 
Let  $v_n \wc 0$ in $ W^{k,q}(\R^d)$, $\ph_1\in  {\cal S}(\R^d)$, $\ph_2 \in {\cal S}(\R^d)$ and let $\psi \in C^\kappa(\textbf{S}^{d-1})$. 
We have to prove  (\ref{kjed}), i.e. (multiplying (\ref{kjed}) by $(i|\xi|)^{-1}$, $|\xi|\neq 0$) that, up to a subsequence,
\begin{equation}\label{tp}
0= \sum_{j=1}^d \left\langle \mu,  A_j \ph_1 \ph_2 \frac{\xi_j}{i|\xi|}\psi\right\rangle = \lim_{n\to\b}  \sum_{j=1}^d \left\langle u_n  A_j \ph_1 \,  , \,  {\cal A}_{\bar\Psi_j}  (\ph_2 v_n)\right\rangle,
\end{equation}
where $\Psi_j= \ds\frac{\xi_j}{i|\xi|} \psi \Big(\frac{\xi}{|\xi|}\Big)$. 
Moreover, 
${\cal A}_{\bar\Psi_j} =-R_j \circ {\cal A}_{\bar\psi}=\pa_j I_1\circ {\cal A}_{\bar\psi},$ see
(\ref{3pote}).  Thus (\ref{tp}) is equivalent to
\begin{equation}\label{tp2}
\lim_{n\to\b} \left\langle \sum_{j=1}^d \pa_j(u_n A_j)\, , \, \bar\ph_1 I_1({\cal A}_{\bar\psi}(\ph_2 v_n))\right\rangle + \sum_{i=1}^d \lim_{n\to\b} \left\langle u_n  A_j \,  , \,   \pa_j(\bar\ph_1) I_1({\cal A}_{\bar\psi}(\ph_2 v_n))\right\rangle = 0.
\end{equation}

Since ${\cal A}_{\bar\psi}(\ph_2 v_n)\in W^{k,q}(\R^d)$ it follows from (\ref{poten}) 
that 
\begin{equation}\label{izvod}
\pa^\alpha I_1  ({\cal A}_{\bar\psi}(\ph_2 v_n)) = I_1  ({\cal A}_{\bar\psi}(\pa^\alpha(\ph_2 v_n))) \in L^{\frac{qd}{d-q}}(\R^d), \mbox{ for all }0\leq |\alpha|\leq k.
\end{equation}
Now, since $\ds q< \frac{qd}{d-q}$, we have that  for all $\vp\in{\cal S}(\R^d)$,
\begin{equation}\label{ogra}
\| \vp\,I_1  ({\cal A}_{\bar\psi}(\ph_2 v_n))\|_{L^q}\leq \| I_1  ({\cal A}_{\bar\psi}(\ph_2 v_n))\|_{L^{\frac{qd}{d-q}}} \, \|\vp\|_{L^d}.
\end{equation}
From (\ref{izvod}) and (\ref{ogra}) we see that
$$
\pa^\alpha [ \vp\,I_1  ({\cal A}_{\bar\psi}(\ph_2 v_n)) ] \in L^q(\R^d),   \mbox{ for all }0\leq |\alpha|\leq k.
$$
Now,
$$
\pa^{\alpha+e_j}  [ I_1  ({\cal A}_{\bar\psi}(\ph_2 v_n)) ] = -R_j({\cal A}_{\bar\psi}(\pa^\alpha(\ph_2 v_n))) \in L^q(\R^d),
$$
which gives us that for all $\vp\in{\cal S}(\R^d)$,
$$
\vp\,I_1  ({\cal A}_{\bar\psi}(\ph_2 v_n))  \in W^{k+1,q}(\R^d),
$$
and moreover 
\begin{equation}\label{slabak}
\vp\,I_1  ({\cal A}_{\bar\psi}(\ph_2 v_n)) \wc 0 \mbox{ in } W^{k+1,q}(\R^d).
\end{equation}

Take $\bar\vp_1=\bar\vp_{11}\bar\vp_{12}$ all in ${\cal S}(\R^d)$.  From 
 (\ref{jedjed}) and (\ref{slabak})  we conclude that
$$
\left\langle \vp_{11} f_n, \bar\vp_{12}\,I_1  ({\cal A}_{\bar\psi}(\ph_2 v_n))  \right\rangle \to 0 .
$$
From here and (\ref{jed}) we conclude that the first term in (\ref{tp2}) converges to zero.

Now we analyze the second term in (\ref{tp2}). We will prove that 
$\pa_j(\bar\ph_1) I_1({\cal A}_{\bar\psi}(\ph_2 v_n))$    converges strongly to zero in $W^{k,q}(\R^d).$
For that purpose we write $\pa_j\bar\ph_1=\bar\ph_{13}\bar\ph_{14}$, all in $\cal S(\R^d)$, and denote by $L_m$ the open ball centered at the origin with radius $m\in\N$. By Rellich lemma $W^{k+1,q}(L_m)$ is compactly embedded in $W^{k,q}(L_m)$. 
Since $\bar\ph_{14} I_1  ({\cal A}_{\bar\psi}(\ph_2 v_n))$ weakly converges to zero in $W^{k+1,q}(\R^d)$, by the diagonalization procedure we can extract a subsequence (not relabeled) such that for all $m\in\N$
\begin{equation}\label{loc}
\bar\ph_{14} I_1({\cal A}_{\bar\psi}(\ph_2 v_n))\scon  0 \mbox{  in } W^{k,q}(L_m).
\end{equation}
Take smooth cutoff functions $\chi_m$ 
such that $\chi_m(x)=1$ for $x\in L_m$ and $\chi_m(x)=0$ for $x\in \R^d\backslash L_{m+1}$ and write $\bar\vp_{13} = \chi_m \bar\vp_{13} + (1-\chi_m) \bar\vp_{13}. $
We have that
\begin{eqnarray}\label{prvi}
& \|\bar\vp_{13} \bar\ph_{14} I_1  ({\cal A}_{\bar\psi}(\ph_2 v_n))\|_{W^{k,q}}\leq   \ds\sup_{|\alpha|\leq k, |x|>m} |\pa^\alpha \bar\vp_{13}|  \,\| \bar\ph_{14} I_1  ({\cal A}_{\bar\psi}(\ph_2 v_n))\|_{W^{k,q}} \\
& \label{drugi} \ds + \|\chi_m \bar\vp_{13} \bar\ph_{14} I_1  ({\cal A}_{\bar\psi}(\ph_2 v_n))\|_{W^{k,q}}.
\end{eqnarray}

Let $\eps>0.$ The sequence
 $\bar\ph_{14} I_1  ({\cal A}_{\bar\psi}(\ph_2 v_n))$ is bounded in $W^{k,q}(\R^d)$, i.e. there is $M>0$ such that $\| \bar\ph_{14} I_1  ({\cal A}_{\bar\psi}(\ph_2 v_n))\|_{W^{k,q}} \leq M$. Since $\bar\vp_{13} \in {\cal S}(\R^d)$, there exists $m_0\in\N$ such that for all $m\geq m_0$
$$
\sup_{|\alpha|\leq k, |x|>m} |\pa^\alpha \bar\vp_{13}| <\frac{\eps}{2M}.
$$

Next, from (\ref{loc}) we have that  (\ref{drugi}) goes to zero as $n\to \b$.  So, for given $\eps$, there exists $n_0\in \N$ such that  (\ref{drugi}) is less than $\eps/2$ for all $n\geq n_0$.
Thus the left hand side in (\ref{prvi}) is less than $\eps$ for $n>n_0$, i.e. 
$\pa_j(\bar\ph_1) I_1({\cal A}_{\bar\psi}(\ph_2 v_n))$ converges strongly in $W^{k,q}(\R^d)$ and  (\ref{tp2}) holds, which completes the proof of (\ref{kjed}).

If coefficients $A_j$ are such that $\sum_{j=1}^d A_j(x) \xi_j \neq 0$, $\xi\in{\bf{S}}^{d-1}$, then  Theorem \ref{th7} implies the strong convergence $\theta u_n \scon 0$,  for every $\theta \in {\cal S}(\R^d)$.

\end{proof}

\section{Appendix}

We give here the full description of spaces $\cal{SE}$ and $\cal{SE}'$ introduced in Subsection 2.1. We recall from \cite[Section 3.8.]{Atkinson} the basic properties of Sobolev spaces on the unit sphere with respect to the surface measure $dS^{d-1}$. In the sequel we assume that $d > 2$. 
Let  $\Omega_l=\{x\in{{\bf R}}^d; |x|\in[1-l,1+l]\}$, $0<l<1$, and $k\in{\N}_0$. 
Then  $\phi\in  C^k ({{\bf S}}^{d-1})$ if for some and hence for all $0<l<1$, $\phi^* \in C^k(\Omega_l)$, 
where $\phi^*(x)=\phi(x/|x|)$. 
Moreover, \cite[p. 9]{Atkinson}, $C^k ({{\bf S}}^{d-1})$ is equipped with the norm 
\be\label{normk}
p_{{{\bf S}}^{d-1},k}(\phi)=|\phi|_{C^k({{\bf S}}^{d-1})}=\sup_{|\alpha|\leq k, x\in\Omega_l}|\pa^{\alpha}\phi^*(x)|, 
\ee
and this norm  does not depend on $l\in(0,1)$. Then
$C^\b ({{\bf S}}^{d-1})= \ds\bigcap_{k\in{\N}_0} C^k ({{\bf S}}^{d-1}) $. The completion of  $C^\b ({{\bf S}}^{d-1})$ with respect to the norm 
$$
\|v\|_{H^s({{\bf S}}^{d-1})}=\Big\|\Big( -\Delta^\star + \Big(\frac{d-2}{2}\Big)^2 \Big) ^{s/2} v \Big\|_{L^2({{\bf S}}^{d-1})},
$$
where $\Delta^\star$ is the Laplace-Beltrami operator, is the Sobolev space $H^s({{\bf S}}^{d-1})$,  $s\in {\N}_0$. The case $d=2$,  when $\textbf{S}^1$ is given by $x=\cos \theta$, $y=\sin \theta$, $\theta\in[0,2\pi)$, is the simple one which we do not consider. Note, in this case one can take $-\Delta^\star +1$ instead of $ -\Delta^\star + (({d-2)}/{2})^2$.

Denote by $\{Y_{n,j},\; 1\leq j \leq N_{n,d}, n\in{\N}_0\}$ the orthonormal basis of $L^2({{\bf S}}^{d-1})$ (cf. \cite[p. 121]{Atkinson} or  \cite[Proposition 10.2, p. 92]{Su}), where $N_{n,d}\sim O(n^{d-2}) $, \cite[p. 16]{Atkinson}, is the dimension of the set of independent spherical harmonics $Y_{n,j}$ of order $n$. Then we have 
\begin{equation}\label{hae}
\|v\|_{H^s({{\bf S}}^{d-1})}=\sqrt{\sum_{n=0}^\b \sum_{j=1}^{N_{n,d}} \Big( n+{{d-2}\over 2} \Big) ^{2s} \|v_{n,j}\|^2},
\end{equation}
where $v_{n,j}=\ds\int_{{{\bf S}}^{d-1}} v \overline{Y}_{n,j} \, dS^{d-1}$.

The space $C^\b ({{\bf S}}^{d-1})$, supplied by the sequence of norms (\ref{normk}), $k\in {\N}_0,$ is denoted by ${\cal E} ({{\bf S}}^{d-1}).$ By  the Sobolev lemma for compact manifolds \cite[Theorems 2.20, 2.21 (see also Theorem 2.10)]{Aubin}, explicitly written in \cite[Theorem 7.6, p. 61]{Su}, we have that  
\begin{equation}\label{d0}
{\mathcal E}({{\bf S}}^{d-1}) = \ds\bigcap_{s\in{\N}_0} H^s({{\bf S}}^{d-1}).
\end{equation} 
This is a Fr\'echet space.
Since all elements of $C^\b({{\bf S}}^{d-1})$ are compactly supported,  we also have that
 $\mathcal E({{\bf S}}^{d-1})=\mathcal D({{\bf S}}^{d-1}).$

By \cite[Theorem II.10, 
p. 52]{RS}, 
if we have orthonormal bases $(\psi_n)_{n\in\bf N}$ and  $(\tilde \psi_m)_{m\in\bf N}$  for $L^2(\R^d, dx)$ with  Lebesgue measure $dx$ and
$L^2(\textbf{S}^{d-1}, dS^{d-1})$ with  surface measure $dS^{d-1}$ respectively,  then  $\psi_n(t_1)\tilde{\psi}_m(t_2),$ $(t_1,t_2)\in \R^d \times \textbf{S}^{d-1}$,   is an orthonormal basis for $L^2(\R^d \times \textbf{S}^{d-1}).$

By \cite[Appendix to V.3, p. 141]{RS} (where the case $d=1$ is treated),  one has  that the product of 
one-dimensional harmonic oscillators $N_x=N_1...N_d$, $N_i=x_i^2-(d/dx_i)^2, i=1,...,d,$ and the Hermite basis 
$h_n(x)=h_{n_1}(x_1)...h_{n_d}(x_d), n\in{\N}_0^d,$ of $L^2(\R^d)$ satisfy $$
N_x^kh_n=(2n_1+1)^k...(2n_d+1)^kh_n, \; n\in{\N}_0^d, \; k\in{\N}_0.
$$
Moreover,  $\mathcal S(\R^d)$ is determined by the sequence of norms
\begin{equation}\label{he1}
|||\phi|||_k=||N^k\phi||_{2}=\sum_{n\in{\N}_0^d}(2n_1+1)^{2k}...(2n_d+1)^{2k}|a_n|^2, \quad k\in\N,
\end{equation}
 where   $\phi=\sum_{n\in{\N}^d}a_nh_n\in\mathcal S(\R^d).$
This sequence of norms is equivalent to  the usual one for $\mathcal S(\R^d)$.

Now, we define the space of smooth functions  $\SE$ by the sequence of norms (\ref{d1}). 
By  the quoted Sobolev lemma for compact manifolds \cite{Aubin}  and (\ref{d0}), we have  the next proposition.

\begin{proposition}\label{pr1}
The  family of norms (\ref{d1}) is equivalent to any of the following two families of  norms:
\begin{equation}\label{d2}
p^2_{\R^d\times\Sd, k}(\theta ) = \left( \int_{\R^d\times \Sd}  |N_x^k(\Delta^\star_\xi)^\alpha \pa_x^\beta \, \theta(x,\xi)|^2 \,dx d\xi\right)^{\frac 12 },
\end{equation}
\begin{equation}\label{d3}
p_{    {\bf R}^d \times {{\bf S}}^{d-1}   ,k}(\theta)
= \sup_{(x,\xi)\in     {\bf R}^d  \times     \Omega_l ,    |\alpha+\beta|\leq k} \langle x\rangle^k|\partial_\xi^\alpha\partial_x^\beta\theta^*
(x,\xi)|, 
\end{equation}
where $\theta^*(x, \xi)=\theta(x, \xi/|\xi|)$, $\langle x\rangle^k =(1+|x|^2)^{k/2}$ and the derivatives with respect to $\xi$ are defined as above, with fixed $x$.

In particular, $\SE$ is a Fr\' echet space.
\end{proposition}

Note that $\SE$ induces the $\pi$-topology on $\cal{S}(\R^d)\otimes \cal{E} (\Sd)$, see \cite[Chap. 43]{TrevesTVS} for the $\pi-$topology. 
Since ${\cal S} ({\R}^d)$ is nuclear, the completion $\mathcal S({\bf R}^{d})  \hat{\otimes}   \mathcal E({\bf S}^{d-1})  $ is the same for the $\pi$ and the $\varepsilon$
 topologies, cf. \cite[Part III, Chap. 50, Theorem 50.1, p. 511]{TrevesTVS}. 
 
 \begin{proposition}\label{pr2}
\begin{equation}
\label{broj3}
\mathcal S({\bf R}^{d})  \hat{\otimes}   \mathcal E({\bf S}^{d-1})   =\mathcal{SE} (  {\bf R}^d   \times   {\bf S}^{d-1}   ).
\end{equation}
\end{proposition}
\begin{proof}
Clearly the embedding
$ \mathcal S({\bf R}^{d})  \hat{\otimes}   \mathcal E({\bf S}^{d-1})   \rightarrow\mathcal{SE}(  {\bf R}^d   \times   {\bf S}^{d-1}  )
$
is continuous. Thus for the proof of (\ref{broj3}), it is enough to prove that the left space is dense in the right one.
As for general manifolds, we have that
 $  h_m (x) \times Y_{n,j}(\xi)  ,$ $1\leq j \leq N_{n,d},\;  n,m\in {\bf N}$, is an orthonormal basis for $L^2(    {\bf R}^d  \times {\bf S}^{d-1}    ).$ 
Now, by (\ref{d0}) -- (\ref{d2}), it follows that
\begin{equation}\label{broj1}
\theta(x,\xi)=\sum_{n=0}^\infty\sum_{j=1}^{N_{n,d}}\sum_{m\in{\N}_0^d} a_{n,j,m} h_m (x) Y_{n,j} (\xi)\in \SE
\end{equation}
if and only if for every $r>0$,
\begin{equation}\label{broj2}
\sum_{n=0}^\infty\sum_{j=1}^{N_{n,d}}\sum_{m\in{\N}_0^d} |a_{n,j,m}|^2(1+n^2+|m|^2)^{r}<\infty
\end{equation}
 (cf. \cite[Chapter 9]{Zemanian}).
 Now taking finite sums of the right-hand side of (\ref{broj1}), we obtain that
$ \mathcal S({\bf R}^{d})  \hat{\otimes}   \mathcal E({\bf S}^{d-1}) $ is dense in $\mathcal{SE}(  {\bf R}^d   \times   {\bf S}^{d-1}  )$.
This completes the proof.
\end{proof}

\section*{Acknowledgement}
We would like to thank to referee of this paper for usefull remarks and his/her contribution to quality of the paper. 

We appreciate the support from Serbian Ministry of Education and  Science(project no. 174024.) and Provincial Secretariat for Science and Technological Development(APV114-451-841/2015-01).

{
   J.~Aleksi\'c, S.~Pilipovi\'c and I.~Vojnovi\'c\\
      Department of Mathematics and informatics, Faculty of Science, University of Novi Sad, Trg D. Obradovi\'ca 4, 21000 Novi Sad, \\
   Serbia
   \email{jelena.aleksic@dmi.uns.ac.rs\\
   stevan.pilipovic@dmi.uns.ac.rs\\
   ivana.vojnovic@dmi.uns.ac.rs}}

\end{document}